\documentclass{amsart}

\usepackage{amsfonts}
\usepackage{amsthm}
\usepackage{amstext}
\usepackage{amsmath}
\usepackage{amscd}
\usepackage{amssymb}
\usepackage[mathscr]{eucal}
\usepackage{epsf}
\usepackage{array}
\usepackage{url}
\usepackage[bookmarks,bookmarksopen,bookmarksdepth=2]{hyperref}
\usepackage{mathrsfs}
\usepackage{enumerate}
\usepackage{mathtools}

\makeatletter
\def\@seccntformat#1{\csname the#1\endcsname.\quad}
\makeatother

\theoremstyle{plain}
\newtheorem{theorem}{Theorem}
\newtheorem{proposition}[theorem]{Proposition}

\newtheorem{lemma}[theorem]{Lemma}

\theoremstyle{definition}

\theoremstyle{remark}
\newtheorem{remark}[theorem]{Remark}

\theoremstyle{remark}

\numberwithin{equation}{section}

\newcommand{\mmod}{\operatorname{mod}}

\newcommand{\AAa}{\mathcal{A}}

\newcommand{\HHh}{\mathcal{H}}
\newcommand{\IIi}{\mathcal{I}}
\newcommand{\JJj}{\mathcal{J}}
\newcommand{\KKk}{\mathcal{K}}
\newcommand{\LLl}{\mathscr{L}}
\newcommand{\NNN}{\mathbb N}

\newcommand{\RRR}{\mathbb R}
\newcommand{\ooo}{o}

\newcommand{\eps}{\varepsilon}
\newcommand{\abs}[1]{\lvert#1\rvert}
\newcommand{\card}[1]{\##1}
\newcommand{\upind}[1]{^{(#1)}}

\newcommand{\RP}{\mathcal{R}}
\newcommand{\NLINES}{\operatorname{L}}
\newcommand{\NLINESS}{\operatorname{\tilde{L}}}
\newcommand{\DENS}{\operatorname{d}}
\newcommand{\DENSS}{\operatorname{\tilde{d}}}
\newcommand{\RR}{\operatorname{RR}}
\newcommand{\DET}{\operatorname{DET}}
\newcommand{\LAVG}{\operatorname{LAVG}}
\newcommand{\LMAX}{\operatorname{LMAX}}
\newcommand{\ENTR}{\operatorname{ENT}}
\newcommand{\RATIO}{\operatorname{RATIO}}
\newcommand{\TREND}{\operatorname{TND}}

\begin{document}
\bibliographystyle{abbrv}

\title{Recurrence quantification analysis of the period-doubling sequence}

\author[V. \v Spitalsk\'y]{Vladim\'ir \v Spitalsk\'y}
\address{Slovanet a.s., Z\'ahradn\'icka 151, Bratislava, Slovakia}
\email{vladimir.spitalsky@slovanet.net}
\address{Department of~Mathematics, Faculty of~Natural Sciences,
    Matej Bel University, Tajovsk\'eho~40, Bansk\'a Bystrica, Slovakia}
\email{vladimir.spitalsky@umb.sk}

%


\begin{abstract}
The period-doubling sequence is one of the most well-known aperiodic
$0$-$1$ sequences. In this paper, a complete description of its symbolic recurrence
plot is given, and formulas for asymptotic values of
basic recurrence quantifiers are derived.
\end{abstract}

\maketitle

\thispagestyle{empty}

\section{Introduction}\label{S:intro}

Recurrence plots \cite{eckmann1987recurrence} provide a visual representation
of recurrences in a trajectory of a dynamical system. Based on them,
recurrence quantification analysis (RQA) introduces new quantitative characteristics
describing complexity of the system \cite{zbilut1992embeddings,webber2015recurrence}.
Several of the mostly used ones, among them recurrence rate ($\RR$), determinism ($\DET$),
average line length ($\LAVG$), and entropy of line lengths ($\ENTR$), are defined via
the so-called diagonal lines in recurrence plots; for the corresponding definitions,
see Section~\ref{SS:RQA}.

Though initially RQA was used for continuous-state dynamics, it can be successfully applied
also to trajectories of discrete-state dynamical systems,
that is, to sequences over a finite alphabet.
In this context, symbolic recurrence plots were proposed in \cite{faure2010recurrence},
see also \cite{faure2015estimating}.
Instead of depending on two parameters: embedding dimension $m$ and distance threshold $\eps$,
symbolic recurrence plots depend only on embedding dimension $m$
(in fact, dependence on $\eps$ can be transformed into dependence on $m$,
see Remark~\ref{R:eps-dependence}).
Further, diagonal lines of any length $\ell$ in the symbolic recurrence plot
with embedding dimension $m$ correspond, in a one-to-one way,
to those of length $\ell+m-1$ in the recurrence plot with embedding dimension $1$.
Thus, instead of recurrence plots depending on $m$ and $\eps$,
recurrence analysis of a symbolic
sequence $x=x_1x_2\dots x_n$ can be based on one symbolic recurrence plot
$\RP(n)=(R_{ij})_{ij=1}^n$ defined simply by $R_{ij}=1$ if $x_i=x_j$, and $R_{ij}=0$
if $x_i\ne x_j$.
Note also that any diagonal line correspond to a (maximal, non-prolongable)
repetition of a subword $w$ of $x$;
so we may say that the diagonal line is determined by the word $w$.
Hence recurrence quantifiers are closely related to combinatorial properties of $x$.

Since our aim is to study asymptotic values of recurrence quantifiers (that is, limits
as $n\to\infty$), we consider an infinite symbolic sequence $x=x_1x_2\ldots$
and its infinite recurrence plot $\RP=(R_{ij})_{i,j=1}^\infty$.
If the sequence $x$ is periodic,
its infinite recurrence plot is very simple:
all diagonal lines have infinite lengths, they begin at the boundary of $\RP$,
and their starting points
are spread evenly; thus all recurrence quantifiers can be easily derived.
Analogously for eventually periodic sequences. On the other hand, if $x$ is not eventually
periodic then, apart from the main diagonal, every diagonal line in the recurrence plot
has finite length.

One of the most well-known aperiodic (but almost periodic, even regularly recurrent)
sequences is the period-doubling
sequence $x=x_1x_2\ldots=0100\,0101\,0100\,0100\,\ldots$.
There are several possible definitions of it; we recall three
of them. First, as the name suggests, the $i$-th member $x_i$ of the sequence
is equal to $k_i\mmod 2$, where $k_i$ is the largest integer such that $2^{k_i}$ divides $i$.
Second, the period-doubling sequence is a unique fixed point of the so-called period-doubling
substitution, that is, the substitution $\xi$ over alphabet $\AAa=\{0,1\}$ given by
$\xi(0)=01$ and $\xi(1)=00$. Third, $x$ is a Toeplitz sequence
given by patterns $(0*)$ and $(1*)$; see \cite{jacobs1969toeplitz}
and \cite[Example~10.1]{downarowicz2005survey}.
For yet another definition of $x$ as the kneading sequence of an interval map, see
\cite[1.10.1]{kurka2003topological}.
The period-doubling sequence and the induced subshift have been studied since 1940s,
see \cite{garcia1948structure} or \cite[12.52]{gottschalk1955topological}.
For a thorough treatment we refer the reader to
\cite{kurka2003topological} (there, the terms Feigenbaum
sequence / subshift are used instead). See also
\cite{damanik2000local,avgustinovich2006sequences,coven2008characterization}
for some recent results.

The period-doubling sequence is aperiodic, but
it is in a sense very regular.
However, as we will show, behavior of its recurrence quantifiers is far from being trivial.
The purpose of this work is to give explicit formulas for asymptotic values of several main
RQA characteristics. Note that some of the characteristics (for example recurrence rate
or determinism) can be defined via correlation integral \cite{grendar2013strong};
hence the knowledge of the (unique) invariant measure of the period-doubling subshift
allows one to obtain formulas for asymptotic values of them \cite{polakova2018complexity}.

In this paper we follow original definition of RQA quantifiers via diagonal lines.
We show that the length $\ell$ of any diagonal line must be of the form $2^{k+1}-1$ or
$3\cdot 2^k-1$ for some $k\ge 0$.
Further, we obtain a simple expression for the set of starting points of diagonal lines
of given length $\ell$, which allows us to compute the density of this set in $\NNN^2$.
These results are summarized in Theorem~\ref{T:line-lengths-and-starts}.
To formulate it, put
\begin{equation*}
    M_1=\{i\in\NNN\colon  x_i=1\},
    \qquad
    M_0=\NNN\setminus M_1 = \{i\in\NNN\colon  x_i=0\},
\end{equation*}
and define
\begin{equation*}
\begin{split}
    A &= (2M_1-1)\times(2M_1+1)
      \quad\sqcup\quad
      (2M_1)\times ((4M_1-1)\sqcup(4M_1+1)),
\\
    B &= (2M_1-1)\sqcup (2M_1),
\\
    C &= (2M_1-1)\times(2M_1).
\end{split}
\end{equation*}

\begin{theorem}\label{T:line-lengths-and-starts}
  Let $\RP$ be the (infinite) symbolic recurrence plot of the period-doubling sequence $x$
  with embedding dimension $1$.
  Let $\ell\ge 1$ and $i,j\ge 1$. Then $(i,j)$ is a starting point
  of a diagonal line of length $\ell$ in $\RP$
  if and only if (exactly) one of the following two cases happens
  \begin{enumerate}
    \item $\ell=2^{k+1}-1$ for some $k\ge 0$ and either $(i,j)$ or $(j,i)$ belongs
      to the union of the sets $2^k A - (2^k-1)$ and $[2^k B - (2^k-1)]\times\{1\}$;
    \item $\ell=3\cdot 2^{k}-1$ for some $k\ge 0$ and either $(i,j)$ or $(j,i)$ belongs
      to the set $2^k C - (2^k-1)$;
  \end{enumerate}
  in both cases, every diagonal line of length $\ell$ is determined by the word
  $x_1^\ell=x_1x_2\dots x_\ell$.
  Consequently, the density of (the set of starting points of)
  diagonal lines of length $\ell\ge 1$ in $\NNN^2$ is
  \begin{equation*}
    \DENS_\ell =
    \frac{c_\ell}{(\ell+1)^2}\,,
    \qquad\text{where}\quad
    c_\ell=
    \begin{cases}
      4/9  &\text{if } \ell=2^{k+1}-1 \text{ for some } k\ge 0;
    \\
      1/2  &\text{if } \ell=3\cdot 2^k-1 \text{ for some } k\ge 0;
    \\
      0                    &\text{otherwise}.
    \end{cases}
  \end{equation*}
\end{theorem}

Theorem~\ref{T:line-lengths-and-starts} enables us to determine asymptotic values
of recurrence quantifiers defined via diagonal lines.
In this paper we consider four of them: recurrence rate $\RR^m_\ell$,
determinism $\DET^m_\ell$, average line length $\LAVG^m_\ell$, and
entropy of line lengths $\ENTR^m_\ell$;
there, $m$ is an embedding dimension and $\ell$ is a lower bound for lengths of diagonal lines
(for corresponding definitions, see Section~\ref{SS:RQA}).
Formulas for these quantifiers are summarized in the following theorem.
As our results show, for large $m$ determinism $\DET^m_\ell$ attains
three possible values:
$1$, $5/7$, and $7/10$. Average line length $\LAVG^m_\ell$
is bounded from both sides by increasing linear functions.
A surprisingly simple formula is obtained for the entropy of
diagonal line lengths $\ENTR^m_\ell$, which is always equal to $2\log2$.
To make the notation easier,
for every $\ell\ge 1$ let $k_\ell\ge 0$ denote the smallest integer
such that $3\cdot 2^{k_\ell-1} - 1 < \ell \le 3\cdot 2^{k_\ell} - 1$;
that is, $k_\ell=\lfloor\log_2((\ell+1)/3\rfloor$.
Distinguish two cases:
\begin{enumerate}[(I)]
  \item\label{IT:1-ell} $3\cdot 2^{k_\ell-1} - 1 < \ell \le 2^{k_\ell+1}-1$;
  \item\label{IT:2-ell} $2^{k_\ell+1}-1 < \ell \le 3\cdot 2^{k_\ell} - 1$;
\end{enumerate}
in the first case put $a_\ell=2$, in the second one put $a_\ell=1$.

\begin{theorem}\label{T:RQA}
  For integers $m,\ell\ge 1$ put $\ell'=\ell+m-1$. Then
  \begin{enumerate}
    \item\label{IT:RR-T:RQA}
      $\RR^m_\ell = \dfrac{2a_{\ell'}+3}{9\cdot 2^{k_{\ell'}}}
        - \dfrac{a_{\ell'}}{9\cdot 4^{k_{\ell'}}}\,$;
    \item\label{IT:DET-T:RQA}
      $\DET^m_\ell=\RR^m_\ell / \RR^m_1$ and, for every $\ell\ge 2$,
      there exists a partition $\NNN=A_1\sqcup A_2\sqcup A_3$ of $\NNN$ into infinite subsets such that $A_1$ has density $1$ and
      \begin{equation*}
        \DET^m_\ell=1 \quad\text{if }m\in A_1,
        \qquad
        \lim_{\substack{m\to\infty\\m\in A_2}} \DET^m_\ell=\frac 57\,,
        \qquad\text{and}\qquad
        \lim_{\substack{m\to\infty\\m\in A_3}} \DET^m_\ell=\frac{7}{10}\,;
      \end{equation*}
    \item\label{IT:LAVG-T:RQA}
      $\LAVG^m_\ell
        = \left( 2+ \frac{3}{a_{\ell'}} \right) 2^{k_{\ell'}} - 1
        \,
      $;
      consequently, $\LAVG^1_1=5/2$ and, for $\ell+m-1\ge 2$,
      \begin{equation*}
        \frac{5}{3}(\ell+m-1) +\frac{2}{3}
        \ \le \
        \LAVG^m_\ell
        \ \le \
        \frac{5}{2}(\ell+m-1) -1;
      \end{equation*}
    \item\label{IT:ENTR-T:RQA}
      $\ENTR^m_\ell=2\log2$.
  \end{enumerate}

\end{theorem}

Let us mention some other recurrence quantifiers, which are not covered by Theorem~\ref{T:RQA}.
A formula for ratio, defined by $\RATIO^m_\ell=\DET^m_\ell/\RR^m_1$, can be readily
obtained from Theorem~\ref{T:RQA}.
Asymptotic value of the maximal length of a diagonal line $\LMAX^m_\ell$
is always equal to $\infty$ \cite[Remark~6]{grendar2013strong};
using the first part of Theorem~\ref{T:line-lengths-and-starts}, one
can derive a formula for $\LMAX^m_\ell(n)$, the maximal length of diagonal lines
in finite recurrence plot $\RP(n)$ of size $n\times n$.

We have not considered recurrence quantifiers defined diagonalwise, i.e.~relatively to a
$\triangle$-diagonal in the recurrence plot
(a $\triangle$-diagonal is the set of pairs $(i,j)$ with $j-i=\triangle$),
for example trend $\TREND^m$ \cite[p.~16]{webber2015recurrence}.
It seems that Theorem~\ref{T:line-lengths-and-starts}
can be useful also for determining these quantifiers.
For example, it implies that every $\triangle$-diagonal
(in the infinite recurrence plot) with $\abs{\triangle}\ge 2$
contains lines of either two or three distinct lengths.

Another topic not considered in this paper concerns recurrence quantifiers
defined via vertical lines (see e.g.~\cite[Section~1.3.2]{webber2015recurrence}).
However, for the period-doubling sequence they are trivial. In fact,
in the recurrence plot with embedding dimension $1$, every   vertical line
has length either $1$ or $3$. Consequently, for embedding dimension greater than $2$, all
vertical lines are singletons.

\medskip

The paper is organized as follows. In the following section we recall basic properties
of the period-doubling sequence as well as definitions of
considered RQA quantifiers. Section~\ref{S:perdoub-segs} is devoted to the proof
of Theorem~\ref{T:line-lengths-and-starts}. Finally, in Section~\ref{S:perdoub-rqa}
we derive formulas for asymptotic values of recurrence quantifiers
and we prove Theorem~\ref{T:RQA}.

\section{Preliminaries}\label{S:preliminaries}
The set of positive (non-negative) integers is denoted by $\NNN$ ($\NNN_0$) and
the set of real numbers is denoted by $\RRR$.
The natural logarithm is denoted by $\log$.
The cardinality of a set $A$ is denoted by $\card{A}$.
We adopt the following conventions. First, $0\log 0=0$. Second, $[a,b]=\emptyset$
for $a>b$.

For a subset $A$ of a Euclidean space and for $a,b\in\RRR$ we put
$aA+b=\{ax+b\colon x\in A\}$.
Let $k\ge 1$ be an integer.
The \emph{(asymptotic) density} of a subset $A$ of $\NNN^k$ is defined by
\begin{equation*}
  d(A) = \lim_{n\to\infty} \frac{1}{n^k} \cdot \card\{(a_1,\dots,a_k)\in A\colon a_i\le n\},
\end{equation*}
provided the limit exists.
Trivially, if both subsets $A$ and $B$ of $\NNN^k$ have density,
and $a\ge 1$ and $b> -a$ are integers,
then $d(A\times B)=d(A)\cdot d(B)$ and $d(aA+b)=a^{-k}d(A)$.
Moreover, if $A,B$ are disjoint then $d(A\sqcup B) = d(A)+d(B)$.

\medskip

Put $\AAa=\{0,1\}$ and
$\AAa^*=\bigcup_{\ell\ge 0} \AAa^\ell$; the members of $\AAa$ and $\AAa^*$ will be called
\emph{letters} and \emph{words}, respectively.
An \emph{$\ell$-word} $w=w_1\dots w_\ell=w_1^\ell$
is any member of $\AAa^\ell$  ($\ell\ge 0$) and
the \emph{length} $\abs{w}$ of it is $\ell$. The unique $0$-word will be denoted by $\ooo$.

Members of $\AAa^\NNN$ will be called \emph{sequences}.
A metric $\varrho$ on $\AAa^\NNN$ is defined by $\varrho(y,z)=2^{-k}$,
where $k=\inf\{i\ge 1\colon y_i\ne z_i\}$.
Note that $(\AAa^\NNN,\varrho)$ is a compact metric space of diameter $1$.
The \emph{(left) shift} on $\AAa^\NNN$ is the map $\sigma\colon \AAa^\NNN\to \AAa^\NNN$
defined by $\sigma(y_1 y_2 \dots)=y_2 y_3 \dots$.

The \emph{concatenation} of (finitely or infinitely many) words, and
of finitely many words and a sequence, is defined in a natural way.

\subsection{The period-doubling sequence}\label{SS:preliminaries-perdoub}
The \emph{period-doubling sequence} is a sequence $x=(x_i)_{i=1}^\infty$, where
$x_i=k_i\mmod 2$ with $k_i\ge 0$ being the largest integer such that $2^{k_i}$ divides
$i$. If we partition $\NNN$ into the sets
$N_k=\{n\in\NNN\colon n\equiv 2^{k}\, (\mmod 2^{k+1})\}$
($k\ge 0$),
then
\begin{equation*}
  M_0=\{i\in\NNN\colon x_i=0\}=\bigsqcup_{k\ge 0} N_{2k}
  \qquad\text{and}\qquad
  M_1=\{i\in\NNN\colon x_i=1\}=\bigsqcup_{k\ge 0} N_{2k+1}.
\end{equation*}
The \emph{language} $\LLl=\LLl_x$ of $x$ is the set of all subwords
$x_i^\ell=x_ix_{i+1}\dots x_{i+\ell-1}$ ($i\ge 1$, $\ell\ge 0$) of $x$.
The orbit closure
of $x$, that is, the closure of the set $\{\sigma^n(x)\colon n\ge 0\}$,
is called the \emph{period-doubling subshift}. It is the set of all sequences
$y$ with the language $\LLl_y$ equal to that of $x$.

The \emph{period-doubling substitution} $\xi\colon\AAa\to\AAa^*$ is defined
by
\begin{equation}\label{EQ:xi-def}
  \xi(a)=0\bar{a}
  \quad (a\in\AAa),
  \qquad\text{where}\quad
  \bar{0}=1 \text{ and }\bar{1}=0.
\end{equation}
It can be naturally extended to $\AAa^*$ and to $\AAa^\NNN$; since no confusion can arise,
these extensions will be denoted again by $\xi$. Iterates of $\xi$ will be denoted
by $\xi^k$ ($k\ge 0$).
The substitution $\xi$ is primitive and has constant length $2$
(for the corresponding notions,
see e.g.~\cite[Chapter~5]{queffelec2010substitution}).
The period-doubling sequence $x$ is the unique fixed point of $\xi$. Thus
\begin{equation}\label{EQ:xi(x)}
  \xi(x_{i})=x_{2i-1}x_{2i}
  \qquad\text{for every } i\ge 1.
\end{equation}

To distinguish even and odd positions in $X$, we will often write the symbol $|$
just before a letter at an odd position. For example, instead of $x=01000101\dots$
we can write $x=|01|00|01|01|\dots$.
We say that an $\ell$-word
$w\in\LLl_x$ ($\ell\ge 1$) is \emph{recognizable} if $(i-j)$ is even
whenever $x_i^\ell=x_j^\ell=w$.
If $w$ is recognizable and some, hence every, $i$ with $x_i^\ell=w$ is even (odd),
we say that the word $w$ is \emph{even (odd)}. So e.g.~for an odd word $w$
we can always write $x=\dots | w \dots$.

\begin{lemma}\label{L:recognizable}
  A word $w\in\LLl_x$ is recognizable if and only if $w\not\in\{\ooo,0,00\}$.
\end{lemma}
\begin{proof}
  The fact that $0$ and $00$ are not recognizable is trivial.
  If $w=000$, for any $i$ with $x_i^3=w$ we have $i\ge 2$ and $x_{i-1}=1$
  (to see it, use that $N_1\subseteq M_1$), hence $i$ is odd
  and $w$ is recognizable.
  Any other nonempty word $w\in\LLl_x$ contains a letter $1$;
  then recognizability follows since $M_1$ contains only even integers.
\end{proof}

\subsection{Recurrence quantification analysis}\label{SS:RQA}
Let $x=(x_i)_{i\ge 1}$ be a (finite or infinite) symbolic sequence over a finite alphabet
$\AAa$. Take integers $n\ge 2$ and $m\ge 1$ such that $n+m-1$ is smaller than or equal to
the length of $x$.
The discrete analogue of $m$-embedding is obtained by considering $m$-words
$x_i^m = x_i x_{i+1}\dots x_{i+m-1}$ ($1\le i\le n$).
Following \cite{faure2010recurrence}, we say that the
\emph{symbolic recurrence plot}\footnote{
  For fixed embedding dimension $m$,
  order patterns recurrence plots \cite{marwan2007recurrence},
  or symbolic recurrence plots
  \cite{caballero2018symbolic},
  form a special case of symbolic recurrence plots as defined in \cite{faure2010recurrence}.
}
is the $n\times n$ matrix
$\RP^m(n) = (R^m_{ij})_{i,j=1}^n$ defined by
\begin{equation*}
  R^m_{ij} =
  \begin{cases}
    1 &\text{if } x_i^m = x_j^m;
  \\
    0 &\text{otherwise}.
  \end{cases}
\end{equation*}
For distinct\footnote{We exclude the main diagonal.}
integers $i,j\in\{1,\dots,n\}$ and for $1\le \ell\le n+1-\max\{i,j\}$,
we say that $(i,j)$ is a \emph{starting point}
and $(i+\ell-1,j+\ell-1)$ is an \emph{end point}
 of a \emph{diagonal line of length $\ell$}
(in the recurrence plot $\RP^m(n)$)
if
$R^m_{i+h,j+h}=1$ for every $0\le h<\ell$, $R^m_{i-1,j-1}=0$ provided $\min\{i,j\}\ge 2$,
and $R^m_{i+\ell,j+\ell}=0$ provided $\max\{i,j\}<=n-\ell$.
The number of diagonal lines of length exactly $\ell$ (at least $\ell$)
in $R^m(n)$ is denoted by $\NLINES^m_\ell(n)$ ($\NLINESS^m_\ell(n)$).
Since we are interested in asymptotics, we prefer to use relative notions,
namely the \emph{frequency of the starting points of diagonal lines} of length
\emph{exactly $\ell$} and of length \emph{at least $\ell$}:
\begin{equation}\label{EQ:density-RP(N)}
  \DENS^m_\ell(n) = \frac{1}{n^2-n} \NLINES^m_\ell(n)
  \qquad\text{and}\qquad
  \DENSS^m_\ell(n) = \frac{1}{n^2-n} \NLINESS^m_\ell(n) = \sum_{l\ge\ell} \DENS^m_l(n).
\end{equation}
Take any $\ell\ge 1$. Then the \emph{recurrence rate},
\emph{determinism}, \emph{average line length},
and \emph{entropy of line lengths} are given by
(see e.g.~\cite[Section~1.3.1]{marwan2015mathematical})
\begin{eqnarray}
  \label{EQ:def-RR}
  \RR^m_\ell(n) &=& \sum_{l \ge \ell} l \DENS^m_{l}(n),
\\
  \label{EQ:def-DET}
  \DET^m_\ell(n) &=&  \frac{\RR^m_\ell(n)}{\RR^m_1(n)},
\\
  \label{EQ:def-LAVG}
  \LAVG^m_\ell(n) &=& \frac{\RR^m_\ell(n)}
      {\DENSS^m_{\ell}(n)},
\\
  \label{EQ:def-ENTR}
  \ENTR^m_\ell(n) &=&
    -\sum_{l\ge\ell}
         \frac{\DENS^m_l(n)}{\DENSS^m_l(n)}
     \cdot\log \frac{\DENS^m_l(n)}{\DENSS^m_l(n)}.
\end{eqnarray}

If $x$ is infinite, taking $n=\infty$ yields the definitions of the
\emph{infinite symbolic recurrence plot} $\RP^m = (R^m_{ij})_{i,j=1}^\infty$ and
the diagonal lines in it. Excluding the trivial case when $x$ is eventually periodic,
we have that all diagonal lines have finite length.
The asymptotic values of recurrence quantifiers are defined by limits as $n$ approaches
infinity, provided the limits exist.
So, for example,
\begin{equation}\label{EQ:density-RP(inf)}
  \DENS^m_\ell = \lim_{n\to\infty} \DENS^m_\ell(n),
  \qquad
  \DENSS^m_\ell = \lim_{n\to\infty} \DENSS^m_\ell(n),
\end{equation}
and analogously for $\RR^m_\ell$, $\DET^m_\ell$, $\LAVG^m_\ell$, and $\ENTR^m_\ell$.
Note that $\DENS^m_\ell$ is equal to the density of the set of starting points
of diagonal lines of length $\ell$ in $\RP^m$.

\begin{remark}[Dependence on embedding dimension $m$]\label{R:embdim-dependence}
    Since
    \begin{equation*}
      \NLINES^m_\ell(n) = \NLINES^{1}_{\ell+m-1}(n)
      \qquad\text{for every } l,m,n\ge 1,
    \end{equation*}
    we have
    $\DENS^m_\ell = \DENS^1_{\ell+m-1}$ and $\DENSS^m_\ell = \DENSS^1_{\ell+m-1}$.
    Thus, in order to determine density $\DENS^m_\ell$,
    we may restrict our considerations to the case when the embedding dimension $m$ equals $1$.
    In such a case we skip the upper index, so we e.g.~write
    $R_{ij}$ instead of $R^1_{ij}$, and
    $\DENS_{\ell}$ instead of $\DENS^1_{\ell}$.
    This is used in Sections~\ref{S:perdoub-segs} and \ref{SS:perdoub-rqa-lemmas}.
\end{remark}

\begin{remark}[Dependence on distance threshold $\eps$]\label{R:eps-dependence}
  Recurrence plots usually depend also on the distance threshold $\eps$.
  However, under the metric $\varrho$ on $\AAa^\NNN$,
  the (continuous) recurrence plot for given $m$ and
  $\eps$ is equal to the symbolic recurrence plot for appropriate embedding dimension $m'$.
  In fact, fix any $\eps >0$ and take a unique integer $h$ such that
  $\eps\in[2^{-h},2^{-h+1})$;  if $h<0$ put $h_\eps=0$
  and otherwise put $h_\eps=h$.
  Then $\varrho(y,z)\le\eps$ is equivalent to $y_i=z_i$ for every $1\le i\le h_\eps$.
  Consequently, the (continuous) recurrence quantifiers for such an $\eps$
  are equal to symbolic ones with embedding dimension equal to $m'=m+h_\eps$;
  for example, $\RR^m_\ell(n,\eps)=\RR^{m+h_\eps}_\ell(n)$
  and $\DET^m_\ell(n,\eps)=\DET^{m+h_\eps}_\ell(n)$.
  Notice that $\eps\to 0$ is equivalent to $h_\eps\to\infty$;
  that is, dependence of recurrence quantifiers
  on the distance threshold $\eps\to 0$ (in the continuous recurrence plot) is
  in fact that on embedding dimension $m\to\infty$ (in the symbolic recurrence plot).
\end{remark}

\section{Lengths and density of diagonal lines for the period-doubling sequence}\label{S:perdoub-segs}
For the period-doubling sequence $x=(x_i)_{i=1}^\infty$
consider the (infinite) {symbolic recurrence plot} $\RP=(R_{ij})_{i,j=1}^\infty$,
where $R_{ij}=1$ if $x_i=x_j$ and $R_{ij}=0$ otherwise;
that is, in the whole section we use the embedding dimension $m$ equal to one; see
Remark~\ref{R:embdim-dependence}.
For $a,b\in\AAa$ and $w\in\LLl_x$ put
\begin{equation}\label{EQ:IJ-def}
  \JJj^{ab}_w = \{i\in\NNN\colon i\ge 2,\ x_{i-1}^{\abs{w}+2} = awb\},
  \qquad
  \IIi^{ab}_w = \JJj^{ab}_w \times \JJj^{\bar{a}\bar{b}}_w,
\end{equation}
and
\begin{equation}\label{EQ:H-def}
  \HHh^b_w = \{i\ge 2\colon x_i^{\abs{w}+1}=wb,\ x_1^{\abs{w}+1}=w \bar{b}\}.
\end{equation}
The sets $\IIi^{ab}_w$ and $\HHh^b_w$
are tightly connected with diagonal lines in the symbolic recurrence plot $\RP$, as
is shown by the following simple result.

\begin{proposition}\label{P:lines-and-IHsets}
  Let $\ell\ge 1$ be an integer and $i,j\ge 1$ be distinct. Then,
  in the (infinite) symbolic recurrence plot $\RP$ of $x$,
  a diagonal line of length $\ell$ starts at $(i,j)$
  if and only if (exactly) one of the following two cases happens:
  \begin{enumerate}
    \item\label{IT:1-L:lines-and-IHsets}
      there are an $\ell$-word $w$ and letters $a,b$ such that
      $(i,j)\in \IIi^{ab}_w$;
    \item\label{IT:2-L:lines-and-IHsets}
      there are an $\ell$-word $w$ and a letter $b$ such that
      $i\in\HHh^b_w$ and $j=1$, or vice versa.
  \end{enumerate}
\end{proposition}
Cases \eqref{IT:1-L:lines-and-IHsets} and \eqref{IT:2-L:lines-and-IHsets}
correspond to diagonal lines starting inside the recurrence plot and at the boundary
of the recurrence plot, respectively.
In both cases we say that the word $w$ \emph{determines} the line starting at $(i,j)$.
\begin{proof}
  The proof is straightforward.
  Take any line of length $\ell$ starting at $(i,j)$ and put $w=x_i^\ell=x_j^\ell$, $b=x_{\ell+i}$.
  Since the length of the line is $\ell$, we must have $x_{\ell+i}\ne x_{\ell+j}$ and
  so $x_{\ell+j}=\bar{b}$.
  If $j=1$ then $i\in\HHh^b_w$; analogously, if $i=1$ then $j\in\HHh^{\bar b}_w$.
  If both $i$ and $j$ are greater than $1$ then put $a=x_{i-1}$ and realize that
  $\bar{a}=x_{j-1}$, so $(i,j)\in\IIi^{ab}_w$.

  On the other hand, trivially any $(i,j)\in\IIi^{ab}_w\sqcup (\HHh^b_w\times\{1\})
  \sqcup (\{1\}\times\HHh^b_w)$ is the starting point of a line of length $\abs{w}$.
\end{proof}

In the rest of the section we determine  all possible lengths $\ell$
of diagonal lines and
we show that for every such length $\ell$ there is a unique word $w$ which determines
all lines of length $\ell$, see Propositions~\ref{P:Iabw} and \ref{P:Hbw}.
We begin with  lines starting inside the recurrence plot.

\subsection{Diagonal lines starting inside the recurrence plot}
The next lemma gives us an easy
recurrent way for determining the sets $\JJj^{ab}_w$.

\begin{lemma}\label{L:J-reduction}
  Let $w\in\LLl_x$ be an odd recognizable word of an odd length.
  Then there is a unique word $\tilde w\in\LLl_x$ such that $w=\xi(\tilde w)0$
  (hence $\abs{w}=2\abs{\tilde w}+1$)
  and
  \begin{equation*}
    \JJj^{ab}_w = 2\JJj^{\bar a\bar b}_{\tilde w} -1
    \qquad\text{for every letters }a,b.
  \end{equation*}
\end{lemma}

\begin{proof}
  Let $\ell$ be such that $\abs{w}=2\ell+1$; then $\ell\ge 1$ by Lemma~\ref{L:recognizable}.
  Take any $i\in\JJj^{ab}_w$. Then $i\ge 2$, $x_{i-1}=a$, $x_i^{2\ell+1}=w$, and
  $x_{i+2\ell+1}=b$. The word $w$ is odd and $i\ne 1$, so there is $h\ge 2$ such that
  $i=2h-1$. By \eqref{EQ:xi(x)} and \eqref{EQ:xi-def},
  \begin{equation*}
    \xi(x_{h-1}^{\ell+2})=\xi(x_{h-1}x_h\dots x_{h+\ell})
    = |x_{2h-3}x_{2h-2}|x_{2h-1}x_{2h}|\dots | x_{2h+2\ell-1} x_{2h+2\ell}|
    = |0a|wb| \,.
  \end{equation*}
  Thus, by \eqref{EQ:xi-def}, $x_{h-1}=\bar a$ and $x_{h+\ell}=\bar b$.
  Further, $w=\xi(\tilde{w})0$, where $\tilde{w}=x_h^{\ell}$.
  So $h\in \JJj^{\bar a\bar b}_{\tilde w}$.
  This proves that $\JJj^{ab}_w \subseteq 2\JJj^{\bar a\bar b}_{\tilde w} -1$.
  The reverse inclusion can be proved analogously.
\end{proof}

The following lemma describes the sets $\JJj^{ab}_w$ for ``short''
words $w$. Recall that $M_1$ is the set of indices $i$ with $x_i=1$.

\begin{lemma}\label{L:J-short-w}
  The following are true:
  \begin{enumerate}
    \item \label{IT:1-L:J-short-w}
      $\JJj^{00}_{\ooo} = (2M_1)\sqcup (2M_1+1)$,
      $\JJj^{11}_{\ooo} = \emptyset$,
      $\JJj^{01}_{\ooo} = M_1$,
      $\JJj^{10}_{\ooo} = M_1+1$;
    \item \label{IT:2-L:J-short-w}
      $\JJj^{00}_{0} = 2M_1$,
      $\JJj^{11}_{0} = (4M_1-1)\sqcup (4M_1+1)$,
      $\JJj^{01}_{0} = 2M_1+1$,
      $\JJj^{10}_{0} = 2M_1-1$;
    \item \label{IT:3-L:J-short-w}
      $\JJj^{00}_{00} = \emptyset$,
      $\JJj^{11}_{00} = \emptyset$,
      $\JJj^{01}_{00} = 2M_1$,
      $\JJj^{10}_{00} = 2M_1-1$.
  \end{enumerate}
\end{lemma}

\begin{proof}
  \eqref{IT:1-L:J-short-w}
  Since $x_i=0$ for every odd $i$, we easily have the last three equalities.
  Take any $i\in\JJj^{00}_{\ooo}$. If $i$ is even, we can write $x_{i-1}x_i=|00|$
  and so $x_{i/2}=1$ and $i/2\in M_1$;
  otherwise $x_{i-1}x_i=0|0$, so $x_{(i-1)/2}=1$ and $(i-1)/2\in M_1$. Thus
  $\JJj^{00}_{\ooo} = (2M_1)\sqcup (2M_1+1)$.

  \eqref{IT:2-L:J-short-w}
  By \eqref{EQ:xi(x)},
  $\JJj^{00}_{0}=\{i\colon x_{i-1}^4=|00|01|\}
  = \{i\colon i \text{ even},\ x_{i/2}^2=10\}
  = 2(\JJj^{10}_\ooo -1)$. This yields the first equality from \eqref{IT:2-L:J-short-w}.
  Analogously,
  $\JJj^{11}_{0}=2J^{00}_\ooo-1$,
  $\JJj^{01}_{0}=2J^{10}_\ooo-1$,
  and $\JJj^{10}_{0}=2J^{01}_\ooo-1$,
  from which the other three equalities follow.

  \eqref{IT:3-L:J-short-w}
  The first two equalities follow from $N_0\subseteq M_0$ and $N_1\subseteq M_1$.
  To prove the other two, it suffices to use the facts that
  $\JJj^{01}_{00}=2(J^{10}_\ooo-1)$ and
  $\JJj^{10}_{00}=2J^{01}_\ooo-1$, which can be obtained as in \eqref{IT:2-L:J-short-w}.
\end{proof}

The next lemma is a direct consequence of Lemma~\ref{L:J-short-w} and \eqref{EQ:IJ-def}.
\begin{lemma}\label{L:I0-I00}
  Let $a,b\in\AAa$. Then
  \begin{enumerate}
    \item $\IIi^{ab}_0\ne\emptyset$;
    \item $\IIi^{ab}_{00}\ne\emptyset$ if and only if $a\ne b$;
    \item $\IIi^{ab}_{w}=\emptyset$ for every $w\not\in\{0,00\}$ of length $1\le \abs{w}\le 2$.
  \end{enumerate}
\end{lemma}

\begin{lemma}\label{L:Iw}
  Let $a,b\in\AAa$ and $w\in\LLl_x$ be recognizable.
  If $\IIi^{ab}_w\ne\emptyset$ then
  $w$ is odd and $\abs{w}\ge 3$ is odd.
\end{lemma}
\begin{proof}
  Put $\ell=\abs{w}$ and take any $(i,j)\in\IIi^{ab}_w$. Hence
  $x_{i-1}^{\ell+2}=awb$ and $x_{j-1}^{\ell+2}=\bar{a}w\bar{b}$; in particular,
  $x_{i-1}=a\ne \bar{a} = x_{j-1}$ and $x_{i+\ell}=b\ne \bar{b} = x_{j+\ell}$.

  Suppose that $w$ is even. Then both $i$ and $j$ are even and
  $x_{i-1}=0=x_{j-1}$, a contradiction.
  Thus $w$ is odd and so both $i$ and $j$ are odd.
  If $\ell$ is even then both $i+\ell$ and $j+\ell$ are odd and so $x_{i+\ell}=0=x_{j+\ell}$,
  a contradiction; thus $\ell$ is odd.

  It suffices to show that $\ell\ge 3$. If this is false then, by
  the previous part of the proof, $\ell=1$ and $w$ is odd, so $w=0$.
  But $w=0$ is not recognizable by Lemma~\ref{L:recognizable}.
  This contradiction shows that $\ell\ge 3$.
\end{proof}

For any word $v=v_1v_2\dots v_\ell$ ($\ell\ge 1$) write $v'=v_1v_2\dots v_{\ell-1}$.

\begin{proposition}\label{P:Iabw}
  Let $a,b\in\AAa$ and $w\in\LLl_x$. Then $\IIi^{ab}_w\ne\emptyset$ if and only if
  there is $k\ge 0$ such that (exactly) one of the following conditions holds:
  \begin{enumerate}
    \item\label{IT:1-P:Iabw}
      $w=x_1^{\ell}$ for $\ell=2^{k+1}-1$;
    \item\label{IT:2-P:Iabw}
      $w=x_1^{\ell}$ for $\ell=3\cdot 2^k-1$, and $a\ne b$.
  \end{enumerate}
  Moreover, if either \eqref{IT:1-P:Iabw} or \eqref{IT:2-P:Iabw} is true then
  \begin{equation*}
    w=[\xi^k(\tilde w 0)]'
    \qquad\text{and}\qquad
    \IIi^{ab}_w=2^k\IIi^{\tilde{a}\tilde{b}}_{\tilde w} - (2^k-1),
  \end{equation*}
  where
  $\tilde w=0$ in \eqref{IT:1-P:Iabw} and $\tilde w=00$ in \eqref{IT:2-P:Iabw},
  $\tilde{a}=a$ and $\tilde{b}=b$ if $k$ is even, and
  $\tilde{a}=\bar a$ and $\tilde{b}=\bar b$ if $k$ is odd.
\end{proposition}
\begin{proof}
  For $\abs{w}\le 2$ the result follows from Lemmas~\ref{L:I0-I00} and \ref{L:Iw};
  so we may assume that $\abs{w}\ge 3$.
  Assume that $\IIi^{ab}_w\ne\emptyset$; we are going to show that there are an integer $k$,
  a word $w\upind{k}\in\{0,00\}$, and letters $a\upind{k},b\upind{k}$ such that
  \begin{equation}\label{EQ:1-P:Iabw}
    w=[\xi^k(w\upind{k}0)]'
    \qquad\text{and}\qquad
    \IIi^{ab}_w = 2^k\cdot \IIi^{a\upind{k}b\upind{k}}_{w\upind{k}} - (2^k-1).
  \end{equation}
  We proceed by induction.
  Put $w\upind{0}=w$, $a\upind{0}=a$, $b\upind{0}=b$, and $\IIi\upind{0}=\IIi^{ab}_w$.
  Assume that, for some $h\ge 0$, $w\upind{h}$, $a\upind{h}$, $b\upind{h}$,
  and $\IIi\upind{h}\ne\emptyset$ have been defined.
  If $\abs{w\upind{h}}\le 2$, put $k=h$ and finish the induction.
  Otherwise, by Lemmas~\ref{L:Iw} and \ref{L:recognizable},
  $w\upind{h}$ is an odd recognizable word
  of odd length. So, by Lemma~\ref{L:J-reduction},
  there is a word $w\upind{h+1}$
  such that
  \begin{equation}\label{EQ:2-P:Iabw}
    w\upind{h}=\xi(w\upind{h+1})\,0
    \qquad\text{and}\qquad
    \IIi\upind{h}=2\IIi\upind{h+1}-1,
  \end{equation}
  where
  $a\upind{h+1}=\overline{a\upind{h}}$,
  $b\upind{h+1}=\overline{b\upind{h}}$,
  and $\IIi\upind{h+1}= \IIi^{a\upind{h+1}b\upind{h+1}}_{w\upind{h+1}}$.
  Since $\abs{w\upind{h+1}}<\abs{w\upind{h}}$,
  the induction always finishes at some finite step and so $k$ is well defined.
  Moreover, $w^{(k)}\in\{0,00\}$ by Lemma~\ref{L:I0-I00} and the fact that
  $\IIi\upind{k}\ne\emptyset$.

  By \eqref{EQ:2-P:Iabw},
  $w=w\upind{0}=\xi^k(w\upind{k}) \xi^{k-1}(0)\xi^{k-2}(0)\dots \xi^{1}(0)0$.
  Since $0=\xi^1(0)'$ and, for every $h\ge 1$,
  $\xi^h(0)[\xi^h(0)]'=[\xi^h(00)]'=[\xi^{h+1}(1)]'=[\xi^{h+1}(0)]'$
  (the last equality follows from the fact that $\xi^{h+1}(0)$ and $\xi^{h+1}(1)$
  differs only at the final letter), we obtain \eqref{EQ:1-P:Iabw}.
  Now it suffices to put $\tilde{w}=w\upind{k}$,
  $\tilde{a}=a\upind{k}$,
  and $\tilde{b}=b\upind{k}$.

  To finish the proof we need to show that either of \eqref{IT:1-P:Iabw} or \eqref{IT:2-P:Iabw}
  implies $\IIi^{ab}_w\ne\emptyset$. To this end, assume that
  \eqref{IT:1-P:Iabw} or \eqref{IT:2-P:Iabw} is true. Analogously as above, an application of
  Lemma~\ref{L:J-reduction} yields \eqref{EQ:1-P:Iabw}. Hence $\IIi^{ab}_w\ne\emptyset$
  by Lemma~\ref{L:I0-I00}.
\end{proof}

\subsection{Diagonal lines starting at the boundary of the recurrence plot}
\begin{proposition}\label{P:Hbw}
  Let $b\in\AAa$ and $w\in\LLl_x$. Then $\HHh^{b}_w\ne\emptyset$ if and only if
  there is $k\ge 0$ such that
  $w=x_1^{\ell}$ for $\ell=2^{k+1}-1$, and $b=0$ if $k$ is even and $b=1$ if $k$ is odd.
  Moreover, if this is satisfied then
  \begin{equation*}
    \HHh^b_w = 2^k \HHh^0_0 - (2^k-1)
    \qquad\text{and}\qquad
    \HHh^0_0=(2M_1-1)\sqcup (2M_1).
  \end{equation*}
\end{proposition}
\begin{proof}
  The proof is analogous to that of Proposition~\ref{P:Iabw}.
  Clearly $\HHh^b_w\ne\emptyset$ implies $w=x_1^\ell$ and $\bar b=x_{\ell+1}$
  for some odd $\ell\ge 1$.
  If $\ell=1$ then $w=0$ and $b=0$;
  thus, by Lemma~\ref{L:J-short-w},
  $$
    \HHh^b_w
    = \{i\ge 2\colon x_i^2=00\}
    = J^{00}_\ooo-1 = (2M_1-1)\sqcup (2M_1).
  $$
  If $\ell\ge 3$ then $w$ is recognizable by Lemma~\ref{L:recognizable},
  and hence odd since $w=x_1^\ell$.
  A result analogous to that of Lemma~\ref{L:J-reduction}
  gives a word $\tilde{w}$ such that
  \begin{equation*}
    w=\xi(\tilde{w})0
    \qquad\text{and}\qquad
    \HHh^b_w = 2\HHh^{\bar b}_{\tilde w} -1.
  \end{equation*}
  Now one implication of the lemma follows as in the proof of Proposition~\ref{P:Iabw}.
  The reverse implication can be proved in the same manner as in Proposition~\ref{P:Iabw}.
\end{proof}

\subsection{Density of lines of given length}
For $\ell\ge 1$ denote the set of starting points of diagonal lines
of length $\ell$ by $\KKk_\ell$.
By Propositions~\ref{P:lines-and-IHsets}, \ref{P:Iabw}, and \ref{P:Hbw},
\begin{equation}\label{EQ:K-def}
  \KKk_\ell =
  \begin{cases}
    \Big(\bigsqcup_{a,b\in\AAa} \IIi^{ab}_w\Big)

    \ \sqcup\
    (\HHh^{b_\ell}_w \times\{1\})
    \ \sqcup\
    (\{1\} \times \HHh^{b_\ell}_w)
    &\text{if } \ell\in 2^\NNN-1,
  \\
      \IIi^{01}_w
      \ \sqcup\
      \IIi^{10}_w
    &\text{if } \ell\in 3\cdot 2^{\NNN_0}-1,
  \\
    \emptyset
    &\text{otherwise},
  \end{cases}
\end{equation}
where $w=x_1^\ell$, and $b_\ell=0$ if $\ell=2^{2k-1}-1$ and $b_\ell=1$ if $\ell=2^{2k}-1$
($k\ge 1$).

\begin{proposition}\label{P:density}
  For any integer $\ell\ge 1$, the density of the set $\KKk_\ell$ is
  \begin{equation*}
    \DENS_\ell =
    \begin{cases}
      1 / (9\cdot 4^k)  &\text{if } \ell=2^{k+1}-1 \text{ for some } k\ge 0;
    \\
      1 / (18\cdot 4^k)  &\text{if } \ell=3\cdot 2^k-1 \text{ for some } k\ge 0;
    \\
      0                    &\text{otherwise}.
    \end{cases}
  \end{equation*}
\end{proposition}
\begin{proof}
  Since any nonempty $\IIi^{ab}_w$ has positive density in $\NNN^2$,
  we may ignore the sets $\HHh^{b_\ell}_w$.
  It is an easy exercise to show that $d(M_1)=1/3$.
  Lemma~\ref{L:J-short-w} easily implies that
  $d(\IIi^{ab}_0)=6^{-2}$ for every $a,b$ and
  $d(\IIi^{ab}_{00})=6^{-2}$ for every distinct $a,b$.

  Assume now that $\ell=2^{k+1}-1$ for some $k\ge 0$;
  by Proposition~\ref{P:Iabw}, $\IIi^{ab}_v=\emptyset$
  for every $\ell$-word $v\ne x_1^\ell$ and every $a,b\in\AAa$;
  further, for $w=x_1^\ell$,
  \begin{equation*}
    d(\IIi^{ab}_w)=\frac{1}{(6\cdot 2^k)^2} \qquad\text{for every }a,b\in\AAa.
  \end{equation*}
  Since the four sets $\IIi^{ab}_w$ ($a,b\in\AAa$)
  are pairwise disjoint, we have $\DENS_\ell=4/(6\cdot 2^k)^2$.

  If $\ell=3\cdot 2^{k}-1$ for some $k\ge 0$, we analogously obtain
  \begin{equation*}
    d(\IIi^{ab}_w)=\frac{1}{(6\cdot 2^k)^2}
    \qquad\text{for } w=x_1^\ell\text{ and every }a\ne b,
  \end{equation*}
  and so $\DENS_\ell=2/(6\cdot 2^k)^2$.

  If $\ell\not\in\{2^{k+1}-1, 3\cdot 2^{k}-1\colon k\ge 0\}$ then
  $\IIi^{ab}_w=\emptyset$ for every $\ell$-word $w$ and every $a,b$; thus
  $\DENS_\ell=0$.
\end{proof}

\subsection{Proof of Theorem~\ref{T:line-lengths-and-starts}}
Now we are ready to prove Theorem~\ref{T:line-lengths-and-starts}.

\begin{proof}[Proof of Theorem~\ref{T:line-lengths-and-starts}]
  By \eqref{EQ:K-def} and Proposition~\ref{P:density}, it suffices to show that
  every nonempty set $\KKk_\ell$ is of the given form. Take any
  $\ell\in(2^\NNN-1)\sqcup (3\cdot 2^{\NNN_0}-1)$ and put $w=x_1^\ell$.
  To make the notation easier, for any subset $S$ of $\NNN^2$
  write $\tilde{S}=\{(j,i)\colon (i,j)\in S\}$.

  Assume first that $\ell=3\cdot 2^k-1$ for some $k\ge 0$.
  For $a\ne b$ Proposition~\ref{P:Iabw} yields
  $\IIi^{ab}_w=2^k\IIi^{\tilde a\tilde b}_{00}-(2^{k-1}-1)$,
  where $\tilde a\ne\tilde b$. Hence, by Lemma~\ref{L:J-short-w},
  $\KKk_\ell
    = 2^k \Big(\IIi^{01}_{00} \sqcup \IIi^{10}_{00}\Big) - (2^k-1)
    = 2^k (C\sqcup \tilde C) - (2^k-1)
  $.
  Now assume that $\ell=2^{k+1}-1$ ($k\ge 0$).
  Propositions~\ref{P:Iabw} and \ref{P:Hbw} give
  $\IIi^{ab}_w = 2^k \IIi^{\tilde a\tilde b}_{0} - (2^k-1)$
  for any $a,b\in\AAa$,
  and
  $\HHh^{b_\ell}_w=2^k B-(2^k-1)$.
  Since
  $
    \bigsqcup_{\tilde a,\tilde b\in\AAa} \IIi^{\tilde a\tilde b}_0
    \ = \
    A\sqcup\tilde{A}
  $
  by Lemma~\ref{L:J-short-w},
  the proof is finished.
\end{proof}

\section{RQA measures for the period-doubling sequence}\label{S:perdoub-rqa}

\subsection{Technical lemmas}\label{SS:perdoub-rqa-lemmas}

In order to derive formulas for asymptotic recurrence quantifiers,
we will need upper
and lower bounds for cardinalities of the sets $\KKk_\ell\cap[1,n]^2$,
which are tightly connected
with densities $\DENS_\ell(n)$ defined in \eqref{EQ:density-RP(N)}.
This is covered by the following two lemmas.

\begin{lemma}\label{L:dens(n)-nonempty}
  For every integers $\ell,n\ge 1$,
  \begin{equation}\label{EQ:1-L:dens(n)-nonempty}
    \KKk_\ell\cap[1,n]^2 \ne \emptyset
    \qquad\text{if and only if}\qquad
    n\ge \ell+2 \text{ and }
    \ell\in(2^{\NNN}-1)\sqcup(3\cdot 2^{\NNN_0}-1).
  \end{equation}
\end{lemma}
\begin{proof}
  If $\ell\not\in(2^{\NNN}-1)\sqcup(3\cdot 2^{\NNN_0}-1)$ then
  $\KKk_\ell=\emptyset$ by Theorem~\ref{T:line-lengths-and-starts}.
  Assume that $\ell=2^{k+1}-1$ for some $k\ge 0$, and
  put $i_0=\ell+2$, $j_0=1$.
  Then, by Theorem~\ref{T:line-lengths-and-starts},
  $(i_0,j_0)\in\KKk_\ell$ and $\max\{i,j\}\ge i_0$ for every $(i,j)\in\KKk_\ell$
  (indeed, it suffices to use $\min M_1=2$). Thus we have \eqref{EQ:1-L:dens(n)-nonempty}
  provided $\ell\in(2^{\NNN}-1)$.
  In the remaining case when $\ell=3\cdot 2^k-1$ ($k\ge 0$) we can prove
  \eqref{EQ:1-L:dens(n)-nonempty} analogously, with $j_0=1$ replaced by
  $j_0=2^{k+1}+1<\ell+2=i_0$.
\end{proof}

\begin{lemma}\label{L:dens(n)}
  There are constants $\alpha,\beta>0$ such that, for every integers $\ell\ge 1$ and $n\ge 2$
  with $\KKk_\ell\cap [1,n]^2\ne\emptyset$,
  \begin{equation}\label{EQ:bounds-L:dens(n)}
    \frac{\alpha}{\ell^2}
    \ \le \
    \frac{1}{n^2-n} \card (\KKk_\ell\cap [1,n]^2)
    \ \le \
    \frac{\beta}{\ell^2}
    \,.
  \end{equation}
\end{lemma}
\begin{proof}
To make the notation easier,
for $h\in\NNN$, a subset $S$ of $\NNN^h$, and $x\in\RRR$ put
\begin{equation}\label{EQ:D(S,x)-def}
  D(S,x) = \card (S\cap [1,x]^h)
\end{equation}
(recall that $[1,x]=\emptyset$ for $x<1$). Clearly,
\begin{equation}\label{EQ:D(S,x)}
  D(S,x)\le x^h
  \qquad\text{for every }x\ge 1
\end{equation}
and, for every integers $c>0$ and $d>-c$,
\begin{equation}\label{EQ:D(cS+d,x)}
  D(cS+d,x) = D(S,(x-d)/c).
\end{equation}

Let $\ell=2^{k+1}-1$ for some $k\ge 0$. Then, by Theorem~\ref{T:line-lengths-and-starts}
and \eqref{EQ:D(cS+d,x)},
\begin{equation}\label{EQ:0-L:dens(n)}
  D(\KKk_\ell,n)
  = 2 D(2^kA-(2^k-1),n) + 2 D(2^kB-(2^k-1),n)
  = 2D(A,\bar n) + 2D(B,\bar n),
\end{equation}
where $\bar n = (n+2^k-1)2^{-k}$.
If $n>2^{k+1}$ then \eqref{EQ:D(S,x)} yields
\begin{equation*}
  D(\KKk_\ell,n)
  \le 2\bar{n}^2+2\bar{n}
  < 2(\bar n+1)^2
  < 2\left( \frac{n}{2^k} + 2\right)^2
  < \frac{8n^2}{4^k}
  = \frac{32 n^2}{(\ell+1)^2}.
\end{equation*}
On the other hand, if $n\le 2^{k+1}$ then $D(\KKk_\ell,n)=0$
by Lemma~\ref{L:dens(n)-nonempty}.
Thus, for every $n\in\NNN$,
\begin{equation}\label{EQ:1-L:dens(n)}
  D(\KKk_\ell,n)
  \le \frac{32 n^2}{(\ell+1)^2}.
\end{equation}
For $\ell=3\cdot 2^{k}-1$ ($k\ge 0$) we analogously obtain
\begin{equation}\label{EQ:1b-L:dens(n)}
  D(\KKk_\ell,n)
  = 2 D(2^kC-(2^k-1),n)
  = 2D(C,\bar n) \le 2\bar{n}^2,
\end{equation}
with $\bar n = (n+2^k-1)2^{-k}$.
Since again $D(\KKk_\ell,n)=0$ if $n\le 2^{k+1}$ by Lemma~\ref{L:dens(n)-nonempty},
for every $n$ we have
\begin{equation}\label{EQ:2-L:dens(n)}
  D(\KKk_\ell,n)
  < \frac{9n^2}{2\cdot 4^k}
  = \frac{81 n^2}{2(\ell+1)^2}.
\end{equation}
Since, by Theorem~\ref{T:line-lengths-and-starts},
$\KKk_\ell=\emptyset$ for every $\ell\not\in(2^{\NNN}-1)\sqcup(3\cdot 2^{\NNN_0}-1)$,
estimates \eqref{EQ:1-L:dens(n)} and \eqref{EQ:2-L:dens(n)}
give the upper bound in \eqref{EQ:bounds-L:dens(n)}.

To obtain the lower bound in \eqref{EQ:bounds-L:dens(n)}, we proceed analogously.
Assume that $D(\KKk_\ell,n)>0$; so $\ell\in\{2^{k+1}-1,3\cdot 2^k-1\}$ for some $k\ge 0$.
Since $D(M_1,x)\ge D(N_1,x)=\lfloor(x+2)/4 \rfloor > (x-2)/4$ for every $x>0$,
\eqref{EQ:D(cS+d,x)} yields that both $D(A,x)=D(2M_1-1,x)\cdot D(2M_1+1,x)$
and $D(C,x)=D(2M_1-1,x)\cdot D(2M_1,x)$ are greater than
$((x-5)/8)^2$. Thus, by \eqref{EQ:0-L:dens(n)} and \eqref{EQ:1b-L:dens(n)},
\begin{equation*}
  D(\KKk_\ell,n) \ge
  2\left(
    \frac{n 2^{-k}-5}{8}
  \right)^2.
\end{equation*}
For $n>10\cdot 2^k$ this gives $D(\KKk_\ell,n) > n^2/(2^5(\ell+1)^2)$.
If $n\le 10\cdot 2^k$ then
\begin{equation*}
  D(\KKk_\ell,n) \ge 1
  \ge \left(  \frac{n}{10\cdot 2^k}  \right)^2
  \ge \frac{4n^2}{10^2 (\ell+1)^2}
  >   \frac{n^2}{2^5 (\ell+1)^2}\,.
\end{equation*}
We have proved that $D(\KKk_\ell,n)>0$
implies $D(\KKk_\ell,n) > n^2/(2^5(\ell+1)^2)$. From this the existence of $\alpha$
readily follows.
\end{proof}

\begin{lemma}\label{L:limsum=sumlim}
  For every integers $m,\ell\ge 1$,
  \begin{equation}\label{EQ:1-L:limsum=sumlim}
  \begin{split}
    &\lim_{n\to\infty} \sum_{l\ge\ell} \DENS^m_l(n)
    =
    \sum_{l=\ell}^\infty  \DENS^m_l,
    \qquad
    \lim_{n\to\infty} \sum_{l\ge\ell} l\DENS^m_l(n)
    =
    \sum_{l=\ell}^\infty l  \DENS^m_l,
    \qquad\text{and}\qquad
  \\
    &\lim_{n\to\infty} \sum_{l\ge\ell} \DENS^m_l(n)\log \DENS^m_l(n)
    =
    \sum_{l=\ell}^\infty \DENS^m_l\log \DENS^m_l.
  \end{split}
  \end{equation}
\end{lemma}
\begin{proof}
  We start by proving the second equality from \eqref{EQ:1-L:limsum=sumlim}.
  By Remark~\ref{R:embdim-dependence} we may assume that $m=1$.
  Let $\alpha,\beta$ be constants from Lemma~\ref{L:dens(n)};
  we may assume that $\alpha<1$.
  For integers $k\ge \ell$ and $n\ge 2$ put (recall the notation \eqref{EQ:D(S,x)-def})
  \begin{equation*}
    \delta_k(n)=\frac{1}{n^2-n} D(\KKk_k,n),
    \qquad
    \eps_k(n)=\sum_{l=\ell}^k l \delta_l(n),
    \qquad
    \eps_k=\sum_{l=\ell}^k l \DENS_l,
    \qquad
    \eps(n)=\sum_{l=\ell}^\infty l \delta_l(n).
  \end{equation*}
  Using the fact that $\lim_n \delta_l(n)=\lim_n \DENS_l(n)=\DENS_l$ we obtain
  \begin{equation*}
    \lim_{n\to\infty} \eps_k(n) = \eps_k.
  \end{equation*}
  Define $\gamma_l=\beta/l$ if $l\in(2^\NNN-1)\sqcup(3\cdot 2^{\NNN_0}-1)$
  and $\gamma_l=0$ otherwise.
  Since $\sum_{l=\ell}^\infty \gamma_l <\infty$ and, by
  Theorem~\ref{T:line-lengths-and-starts} and Lemma~\ref{L:dens(n)},
  $0\le l \delta_l(n)\le\gamma_l$ for every $l\ge\ell$,
  Weierstrass M-test yields
  \begin{equation*}
    \lim_{k\to\infty} \eps_k(n) = \eps(n)
    \qquad\text{uniformly in }n.
  \end{equation*}
  Now, by Moore-Osgood theorem (see e.g.~\cite[p.~140]{taylor1985general}),
  $\lim_n\lim_k\eps_k(n) = \lim_k\lim_n\eps_k(n)$, that is,
  \begin{equation*}
    \lim_{n\to\infty} \sum_{l\ge\ell} l \delta_l(n)
    = \sum_{l=\ell}^\infty l \DENS_l.
  \end{equation*}
  Thus, to finish the proof of the second equality from \eqref{EQ:1-L:limsum=sumlim}
  it suffices to show that
  $\sum_{l\ge\ell} l(\delta_l(n) - \DENS_l(n))$ converges to zero as $n\to\infty$.

  We say that a diagonal line with starting point $(i,j)$ and length $l$
  (in the infinite recurrence plot $\RP$) is
  \emph{$n$-boundary} if $n-l < \max(i,j)\le n$
  (that is, the line starts in $\RP(n)$ and
  contains a recurrence with some coordinate equal to $n$:
  $i+h=n$ or $j+h=n$ for some $0\le h<l$). Denote the number of
  $n$-boundary lines of length $l$ by $b_l(n)$, and put
  \begin{equation*}
    S_\ell(n)
    =(n^2-n)
      \sum_{l\ge\ell} l(\delta_l(n) -
                      \DENS_l(n))
    =\sum_{l\ge\ell} l \left(D(\KKk_l,n)-\NLINES_l(n)\right).
  \end{equation*}
  Clearly, $S_\ell(n)$ is non-negative and bounded from above
  by the number of recurrences (in the infinite recurrence plot $\RP$) contained
  in $n$-boundary lines:
  \begin{equation*}
    S_\ell(n) \le
    \sum_{l\ge\ell} l \cdot b_l(n).
  \end{equation*}

  We claim that $l b_l(n)< 6n$ for every $l$.
  Indeed, this is trivially true
  for $l\not\in(2^\NNN-1)\sqcup(3\cdot 2^{\NNN_0}-1)$, so assume
  that $l\in \{2^{k+1}-1,3\cdot 2^k-1\}$ for some $k\ge 0$.
  Theorem~\ref{T:line-lengths-and-starts} implies that, for the starting point $(i,j)$
  of any diagonal line of length $l$, $\abs{i-j}$ is a (non-zero) multiple of $2^k$.
  Hence
  \begin{equation}\label{EQ:2-L:limsum=sumlim}
    b_l(n)\le 2(n-1)/2^k
  \end{equation}
  and so
  $lb_l(n)\le 2(3 \cdot 2^k-1)(n-1)/2^k < 6n$.

  Let $l\ge\ell$ be such that $b_l(n)>0$. Then
  $l\in \{2^{k+1}-1,3\cdot 2^k-1\}$ for some $k\ge 0$
  by Theorem~\ref{T:line-lengths-and-starts},
  and $n>2^{k+1}$ by Lemma~\ref{L:dens(n)}.
  Thus $0\le k<\log_2n-1$ and so there are at most $2\log_2 n$
  different lengths of $n$-boundary lines. We obtained that
  \begin{equation*}
    S_\ell(n)< 12n\log_2 n.
  \end{equation*}
  Hence $\sum_{l\ge\ell} l(\delta_l(n) - \DENS_l(n))$
  converges to zero as $n$ approaches $\infty$,
  which finishes the proof of the second equality from \eqref{EQ:1-L:limsum=sumlim}.

  \medskip
  The proof of the first equality from \eqref{EQ:1-L:limsum=sumlim} is analogous;
  one only needs to replace the definition of $\gamma_l$ by
  $\gamma_l=\beta/l^2$ for every $l\ge\ell$. The fact that
  $\lim_n\sum_{l\ge\ell} \delta_l(n) =\lim_n\sum_{l\ge\ell} \DENS_l(n)$
  can be proved as above.

  \medskip
  Now we show the third equality from \eqref{EQ:1-L:limsum=sumlim}. The estimate
  $0\le -\delta_l(n)\log\delta_l(n)
  \le \beta(2\log l-\log\alpha) / l^2$ for every $l\ge\ell$
  (which is trivially satisfied also for $n$ with $\delta_l(n)=0$)
  and Moore-Osgood theorem give
  \begin{equation*}
    \lim_{n\to\infty} \sum_{l\ge\ell} \delta_l(n)\log\delta_l(n)
    = \sum_{l=\ell}^\infty \DENS_l\log\DENS_l.
  \end{equation*}
  It remains to prove that
  $\lim_{n} \sum_{l\ge\ell} \delta_l(n)\log\delta_l(n)
  =\lim_{n} \sum_{l\ge\ell} \DENS_l(n)\log\DENS_l(n)$.
  To this end, fix any $n$ and put $\psi(x)=-x\log x$ for $x\ge 0$.
  For $l\ge\ell$ define $a_l=D(\KKk_l,n)$ and
  $\triangle_l=(n^2-n)\DENS_l(n)-a_l$;
  note that $a_l,\triangle_l$ are integers.
  In this notation,
  \begin{equation}\label{EQ:2b-L:limsum=sumlim}
  \begin{split}
    -\sum_{l\ge\ell} \delta_l(n)\log\delta_l(n)
    &= \log (n^2-n)
       \ -\  \frac{1}{n^2-n} \sum_{l\ge\ell} \psi(a_l),
    \\
    -\sum_{l\ge\ell} \DENS_l(n)\log\DENS_l(n)
    &= \log (n^2-n)
       \ -\  \frac{1}{n^2-n} \sum_{l\ge\ell} \psi(a_l+\triangle_l).
  \end{split}
  \end{equation}

  Note that $\sum_{l\ge\ell}\abs{\triangle_l}$ is bounded from above by $2\tilde{b}_\ell(n)$,
  where $\tilde{b}_\ell(n)=\sum_{l\ge \ell} b_l(n)$ is the number of all $n$-boundary lines
  of length at least $\ell$. By \eqref{EQ:2-L:limsum=sumlim},
  \begin{equation}\label{EQ:4-L:limsum=sumlim}
    \sum_{l\ge\ell}\abs{\triangle_l}
    \le 4(n-1)\cdot\left(
      \sum_{l=2^{k+1}-1\,\ge\, \ell} 2^{-k}
      +
      \sum_{l=3\cdot 2^{k}-1\,\ge\, \ell} 2^{-k}
    \right)
    < 16n.
  \end{equation}

  Fix any $l\ge\ell$. If both $a_l$ and $a_l+\triangle_l$ are non-zero
  (hence bounded from below by $1$) then, by the mean value theorem,
  there is $c_l$  between $a_l$ and $a_l+\triangle_l$ such that
  $\psi(a_l+\triangle_l)-\psi(a_l)={\triangle_l}(\log c_l+1)$.
  Since trivially both $a_l$ and $a_l+\triangle_l$ are smaller than
  $n^2$, we have
  \begin{equation}\label{EQ:5-L:limsum=sumlim}
    \abs{\psi(a_l+\triangle_l)-\psi(a_l)}
    <
    \abs{\triangle_l} \cdot(1+2\log n).
  \end{equation}
  On the other hand, if $a_l=0$ or $a_l+\triangle_l=0$ then
  trivially
  \begin{equation*}
    \abs{\psi(a_l+\triangle_l)-\psi(a_l)}
    = \abs{\triangle_l} \cdot\log \abs{\triangle_l},
  \end{equation*}
  hence \eqref{EQ:5-L:limsum=sumlim} is true also in this case.
  Now \eqref{EQ:2b-L:limsum=sumlim},
  \eqref{EQ:4-L:limsum=sumlim}, and \eqref{EQ:5-L:limsum=sumlim}
  yield that $\sum_{l\ge\ell} (\delta_l(n)\log\delta_l(n)-\DENS_l(n)\log\DENS_l(n))$
  converges to zero as $n\to\infty$.
\end{proof}

\subsection{Proof of Theorem~\ref{T:RQA}}\label{SS:perdoub-rqa-theorem}
In this section we derive explicit formulas for asymptotic values of recurrence rate,
determinism, average line length, and entropy of line lengths, of the period-doubling sequence;
recall the definitions and notation from Section~\ref{SS:RQA}.
We will use
the following formulas, the easy proofs of
which are omitted:
\begin{equation}\label{EQ:geo-sums}
\begin{split}
  &\sum_{k=h}^\infty 4^{-k} = \frac{1}{3\cdot 4^{h-1}}\,,
\qquad
  \sum_{k=h}^\infty k 4^{-k} = \frac{3h+1}{9\cdot 4^{h-1}}\,,
\\
  &\sum_{k=h}^\infty \frac{a}{4^{k}} \log\frac{a}{4^{k}}
   = \frac{3a\log a - 2a(3h+1)\log2}{9\cdot 4^{h-1}}\,;
\end{split}
\end{equation}
there, $a$ is any positive real number and $h\in\NNN$.
Recall from Introduction that
$k_\ell=\lfloor\log_2((\ell+1)/3\rfloor$ for every $\ell\in\NNN$,
$a_\ell=2$ if $3\cdot 2^{k_\ell-1}-1<\ell\le 2^{k_\ell+1}-1$
(case \eqref{IT:1-ell}), and
$a_\ell=1$ if $2^{k_\ell+1}-1<\ell\le 3\cdot 2^{k_\ell}-1$
(case \eqref{IT:2-ell}).

\begin{lemma}\label{L:dens-rr-entr}
  Let $\ell\ge 1$. Then
  \begin{eqnarray*}
    \sum_{l\ge\ell} \DENS_l
    &=&
    \frac{a_\ell}{9\cdot 4^{k_\ell}}\,;
  \\
    \sum_{l\ge\ell} l\DENS_l
    &=&  \frac{2a_\ell+3}{9\cdot 2^{k_\ell}} - \frac{a_\ell}{9\cdot 4^{k_\ell}}\,;
  \\
    -\sum_{l\ge \ell} \DENS_l \log\DENS_l
    &=&  \frac{(a_\ell k_\ell+1)\log2 + a_\ell \log3}{18\cdot 4^{k_\ell-1}}\,.
  \end{eqnarray*}
\end{lemma}
\begin{proof}
  To prove these equalities it suffices to use \eqref{EQ:geo-sums} and
  Proposition~\ref{P:density}, and to realize that $\DENS_l$ ($l\ge\ell$) is non-zero
  only in the following two cases:
  first, if $l=2^{k+1}-1$ for some $k\ge k'$,
  where $k'=k_\ell$ in case \eqref{IT:1-ell}, and
  $k'=k_\ell+1$ in case \eqref{IT:2-ell};
  second, if $l=3\cdot 2^{k}-1$ for some $k\ge k_\ell$.
\end{proof}

\begin{proposition}[Recurrence rate]\label{P:RR}
  For every $m,\ell\ge 1$ we have
  \begin{equation*}
    \RR^m_\ell =
    \frac{2a_{\ell'}+3}{9\cdot 2^{k_{\ell'}}} - \frac{a_{\ell'}}{9\cdot 4^{k_{\ell'}}}\,
  \end{equation*}
  where $\ell'=\ell+m-1$.
\end{proposition}
\begin{proof}
  For $m=1$ the formula follows from \eqref{EQ:def-RR} and
  Lemmas~\ref{L:limsum=sumlim}, \ref{L:dens-rr-entr}.
  For general $m$ use Remark~\ref{R:embdim-dependence}.
\end{proof}

The following two propositions give formulas for the (asymptotic) determinism and
average line length.

\begin{proposition}[Determinism]\label{P:DET}
  For every $m,\ell\ge 1$ we have
  \begin{equation*}
    \DET^m_\ell = \frac{\RR^m_\ell}{\RR^m_1}\,.
  \end{equation*}
  Consequently, for every $\ell\ge 2$
  there exists a partition $\NNN=A_1\sqcup A_2\sqcup A_3$ of $\NNN$ into infinite subsets such that $A_1$ has density $1$ and
  \begin{equation*}
    \DET^m_\ell=1 \quad\text{if }m\in A_1,
    \qquad
    \lim_{\substack{m\to\infty\\m\in A_2}} \DET^m_\ell=\frac 57\,,
    \qquad\text{and}\qquad
    \lim_{\substack{m\to\infty\\m\in A_3}} \DET^m_\ell=\frac{7}{10}\,.
  \end{equation*}
\end{proposition}
\begin{proof}
  The fact that $\DET^m_\ell = {\RR^m_\ell}/{\RR^m_1}$ immediately follows from
  \eqref{EQ:def-DET}.
  Fix $\ell\ge 2$ and put
  \begin{equation*}
    A_1=\{m\in\NNN\colon k_{\ell+m-1}=k_m, \ a_{\ell+m-1}=a_m\};
  \end{equation*}
  that is, $A_1$ is the set of all integers $m$ for which there is $k\ge 0$ such that
  either $3\cdot 2^{k-1}-1 < m\le m+\ell-1\le 2^{k+1}-1$
  or $ 2^{k+1}-1 < m\le m+\ell-1\le  3\cdot 2^{k}-1$.
  Clearly, the density of $A_1$ is $d(A_1)=1$ and, by Proposition~\ref{P:RR},
  $\DET^m_\ell=1$ for every $m\in A_1$.

  Define
  \begin{eqnarray*}
     A_2
     &=&
     \{m\in\NNN\colon k_{\ell+m-1}=k_\ell, \ a_m=2,\ a_{\ell+m-1}=1\}
  \\
     &=&
     \{
       m\in\NNN\colon
       3\cdot 2^{k-1}-1 < m\le 2^{k+1}-1 < \ell+m-1\le 3\cdot 2^{k}-1
       \text{ for some }k\ge 0
     \};
  \\
     \tilde{A}_3
     &=&
     \{m\in\NNN\colon k_{\ell+m-1}=k_\ell+1, \ a_m=1,\ a_{\ell+m-1}=2\}
  \\
     &=&
     \{
       m\in\NNN\colon
        2^{k+1}-1 < m\le 3\cdot 2^{k}-1 < \ell+m-1\le 2^{k+2}-1
       \text{ for some }k\ge 0
     \}.
  \end{eqnarray*}
  Note that the sets $A_1,A_2,\tilde{A}_3$ are pairwise disjoint,
  both $A_2$ and $\tilde{A}_3$ are infinite,
  and $B=\NNN\setminus(A_1\sqcup A_2\sqcup \tilde{A}_3)$ is finite.
  Further, by Proposition~\ref{P:RR},
  \begin{equation*}
    \lim_{\substack{m\to\infty\\m\in A_2}} \DET^m_\ell
    =\frac{2\cdot 1+3}{2\cdot 2+3}
    =\frac 57\,,
  \qquad\text{and}\qquad
    \lim_{\substack{m\to\infty\\m\in \tilde{A}_3}} \DET^m_\ell=\frac{7}{10}
    =\frac{2\cdot 2+3}{2(2\cdot 1+3)}
    =\frac 7{10}\,.
  \end{equation*}
  Thus, taking $A_3=\tilde{A}_3\sqcup B$, the proposition is proved.
\end{proof}

\begin{proposition}[Average line length]\label{P:LAVG}
  For every $m,\ell\ge 1$ we have
    \begin{equation*}
    \LAVG^m_\ell
    = \left( 2+ \frac{3}{a_{\ell'}} \right) 2^{k_{\ell'}} - 1
    \,,
  \end{equation*}
  where $\ell'=\ell+m-1$.
  Consequently, $\LAVG^1_1=5/2$ and, for $\ell+m-1\ge 2$,
  \begin{equation*}
    \frac{5}{3}(\ell+m-1) +\frac{2}{3}
    \ \le \
    \LAVG^m_\ell
    \ \le \
    \frac{5}{2}(\ell+m-1) -1.
  \end{equation*}
\end{proposition}
\begin{proof}
  By Lemmas~\ref{L:limsum=sumlim} and \ref{L:dens-rr-entr} and
  Remark~\ref{R:embdim-dependence},
  $\DENSS^m_\ell=\sum_{l\ge\ell} \DENS^m_l = a_{\ell'}/(9\cdot 4^{k_{\ell'}})$.
  Thus the formula for $\LAVG^m_\ell$ follows from \eqref{EQ:def-LAVG} and
  Proposition~\ref{P:RR}.
  The inequalities can be obtained easily by employing boundaries for $2^{k_{\ell'}}$,
  separately for cases \eqref{IT:1-ell} and \eqref{IT:2-ell}.
\end{proof}

For entropy of line lengths we obtain the following surprisingly simple formula.

\begin{proposition}[Entropy of line lengths]\label{P:ENTR}
  For every $m,\ell\ge 1$ we have
  \begin{equation*}
    \ENTR^m_\ell = 2\log2.
  \end{equation*}
\end{proposition}
\begin{proof}
  We may assume that $m=1$.
  By \eqref{EQ:def-ENTR} and Lemma~\ref{L:limsum=sumlim},
  \begin{equation*}
    \ENTR_\ell=-\sum_{l=\ell}^\infty \frac{\DENS_l}{\DENSS_\ell}\log\frac{\DENS_l}{\DENSS_\ell}
    =\log\DENSS_\ell - \frac{1}{\DENSS_\ell}\sum_{l=\ell}^\infty \DENS_l\log \DENS_l.
  \end{equation*}
  Now it suffices to use the first and the third formulas
  from Lemma~\ref{L:dens-rr-entr}.
\end{proof}

\begin{proof}[Proof of Theorem~\ref{T:RQA}]
  Theorem~\ref{T:RQA} immediately follows from Propositions~\ref{P:RR}, \ref{P:DET},
  \ref{P:LAVG}, and \ref{P:ENTR}.
\end{proof}

\section*{Acknowledgements}
Substantive feedback from Miroslava Pol\'akov\'a
is gratefully acknowledged.
This research is an outgrowth of the project ``SPAMIA'', M\v S SR-3709/2010-11,
supported by the Ministry of Education, Science, Research and Sport of the Slovak Republic,
under the heading of the state budget support for research and development.
The author also acknowledges support from VEGA~1/0786/15 and APVV-15-0439 grants.

\bibliography{rqa-perdoub}

\begin{thebibliography}{10}

\bibitem{avgustinovich2006sequences}
S.~V. Avgustinovich, J.~Cassaigne, and A.~E. Frid.
\newblock Sequences of low arithmetical complexity.
\newblock {\em Theor.~Inform.~Appl.}, 40(4):569--582, 2006.

\bibitem{caballero2018symbolic}
M.~V. Caballero-Pintado, M.~Matilla-Garc{\'\i}a, and M.~Ruiz~Mar{\'\i}n.
\newblock Symbolic recurrence plots to analyze dynamical systems.
\newblock {\em Chaos}, 28(6):063112, 2018.

\bibitem{coven2008characterization}
E.~M. Coven, M.~Keane, and M.~LeMasurier.
\newblock A characterization of the {M}orse minimal set up to topological
  conjugacy.
\newblock {\em Ergodic Theory Dynam.~Systems}, 28(5):1443--1451, 2008.

\bibitem{damanik2000local}
D.~Damanik.
\newblock Local symmetries in the period-doubling sequence.
\newblock {\em Discrete Appl.~Math.}, 100(1-2):115--121, 2000.

\bibitem{downarowicz2005survey}
T.~Downarowicz.
\newblock Survey of odometers and {T}oeplitz flows.
\newblock {\em Contemp.~Math.}, 385:7--38, 2005.

\bibitem{eckmann1987recurrence}
J.~P. Eckmann, S.~O. Kamphorst, and D.~Ruelle.
\newblock {Recurrence plots of dynamical systems}.
\newblock {\em {Europhys.~Lett.}}, {4}({9}):{973--977}, {1987}.

\bibitem{faure2010recurrence}
P.~Faure and A.~Lesne.
\newblock Recurrence plots for symbolic sequences.
\newblock {\em Internat.~J.~Bifur.~Chaos}, 20(06):1731--1749, 2010.

\bibitem{faure2015estimating}
P.~Faure and A.~Lesne.
\newblock Estimating {K}olmogorov entropy from recurrence plots.
\newblock In {\em Recurrence Quantification Analysis}, pages 45--63. Springer,
  2015.

\bibitem{garcia1948structure}
M.~Garcia and G.~A. Hedlund.
\newblock The structure of minimal sets.
\newblock {\em Bull.~Amer.~Math.~Soc.}, 54(10):954--964, 1948.

\bibitem{gottschalk1955topological}
W.~H. Gottschalk and G.~A. Hedlund.
\newblock {\em Topological dynamics}, volume~36.
\newblock American Mathematical Society, 1955.

\bibitem{grendar2013strong}
M.~Grend{\'a}r, J.~Majerov{\'a}, and V.~{\v{S}}pitalsk{\'y}.
\newblock Strong laws for recurrence quantification analysis.
\newblock {\em Internat.~J.~Bifur.~Chaos}, 23(08):1350147, 2013.

\bibitem{jacobs1969toeplitz}
K.~Jacobs and M.~Keane.
\newblock 0-1-sequences of {T}oeplitz type.
\newblock {\em Z.~Wahrscheinlichkeitstheorie und Verw.~Gebiete},
  13(2):123--131, 1969.

\bibitem{kurka2003topological}
P.~K{\uu}rka.
\newblock {\em Topological and symbolic dynamics}, volume~11.
\newblock SMF, 2003.

\bibitem{marwan2007recurrence}
N.~Marwan, M.~C. Romano, M.~Thiel, and J.~Kurths.
\newblock {Recurrence plots for the analysis of complex systems}.
\newblock {\em {Phys.~Rep.}}, {438}({5-6}):{237--329}, {2007}.

\bibitem{marwan2015mathematical}
N.~Marwan and C.~L. Webber.
\newblock Mathematical and computational foundations of recurrence
  quantifications.
\newblock In {\em Recurrence Quantification Analysis}, pages 3--43. Springer,
  2015.

\bibitem{polakova2018complexity}
M.~Pol\'akov\'a.
\newblock Complexity and invariant measure of the period-doubling subshift.
\newblock Work in progress, 2018.

\bibitem{queffelec2010substitution}
M.~Queff{\'e}lec.
\newblock {\em Substitution dynamical systems---spectral analysis}, volume
  1294.
\newblock Springer, 2010.

\bibitem{taylor1985general}
A.~E. Taylor.
\newblock {\em General theory of functions and integration}.
\newblock Courier Corporation, 1985.

\bibitem{webber2015recurrence}
C.~L. Webber~Jr and N.~Marwan.
\newblock {\em Recurrence quantification analysis: theory and best practices}.
\newblock Springer, 2015.

\bibitem{zbilut1992embeddings}
J.~P. Zbilut and C.~L. Webber.
\newblock {Embeddings and delays as derived from quantification of recurrence
  plots}.
\newblock {\em {Phys.~Lett.~A}}, {171}({3-4}):{199--203}, {1992}.

\end{thebibliography}

\end{document}